\documentclass[12pt]{amsart}
\usepackage{amsmath, amsthm}
\usepackage[english]{babel}
\usepackage[T1]{fontenc}
\usepackage[latin1]{inputenc}
\usepackage{amssymb}
\usepackage{amscd}
\usepackage{latexsym}
\usepackage[bookmarksnumbered,plainpages,hypertex]{hyperref}
\usepackage{graphicx}
\usepackage{mathrsfs} 

\newtheoremstyle{theorem}
  {12pt}          
  {12pt}  
  {\sl}  
  {\parindent}     
  {\bf}  
  {. }    
  { }    
  {}     
\theoremstyle{theorem}
\newtheorem{theorem}{Theorem}[section]
\newtheorem{corollary}[theorem]{Corollary}
\newtheorem{remark}[theorem]{Remark}
\newtheorem{proposition}[theorem]{Proposition}
\newtheorem{lemma}[theorem]{Lemma}

\newtheorem{definition}[theorem]{Definition}
\newtheorem{notations}[theorem]{Notations}

\linethickness{0.5pt}

\newcommand{\ic}{\ensuremath{\mathcal{I}}}

\newcommand{\oc}{\ensuremath{\mathcal{O}}}

\newcommand{\lc}{\ensuremath{\mathcal{L}}}

\newcommand{\mc}{\ensuremath{\mathcal{M}}}

\newcommand{\Pt}{\mathbb{P}^3}
\newcommand{\Ptw}{\mathbb{P}^2}

\newcommand{\Pcq}{\mathbb{P}^5}

\newcommand{\Pn}{\mathbb{P}^n}

\newcommand{\bN}{\mathbb{N}}

\newcommand{\bC}{\mathbb{C}}

\newcommand{\bP}{\mathbb{P}}

\newcommand{\aG}{\alpha}

\newcommand{\fG}{\varphi}

\newcommand{\sG}{\sigma}

\newcommand{\lag}{\langle}
\newcommand{\rag}{\rangle}

\newcommand{\bds}{\begin{displaystyle}}
\newcommand{\eds}{\end{displaystyle}}

\title[Chern classes.]{Chern classes of rank two globally generated vector bundles on $\Ptw$.}

\author{Ph. Ellia}
\address{Dipartimento di Matematica, 35 via Machiavelli, 44100 Ferrara}
\email{phe@unife.it}


\begin{document}
\maketitle

\thispagestyle{empty}

\begin{abstract} We determine the Chern classes of globally generated rank two vector bundles on $\Ptw$.
\end{abstract}

\section*{Introduction.}

\par
Vector bundles generated by global sections come up in a variety of problems in projective algebraic geometry. In this paper we consider the following question: \emph{which are the possible Chern classes of rank two globally generated vector bundles on $\Ptw$?} (Here $\Ptw =\Ptw _k$ with $k$ algebraically closed, of charateristic zero.) 
\par
Clearly these Chern classes have to be positive. Naively one may think that this is the only restriction. A closer inspection shows that this is not true: since we are on $\Ptw$, the construction of rank two vector bundles starting from codimension two, locally complete intersection subschemes is subject to the Cayley-Bacharach condition (see Section \ref{S:CB}). So if we have an exact sequence $0 \to \oc \to F \to \ic _Y(c) \to 0$, with $F$ a rank two vector bundle and $Y \subset \Ptw$ of codimension two, then $Y$ satisfies Cayley-Bacharach for $c-3$.
\par
Now $F$ is globally generated if and only if $\ic _Y(c)$ is. If $Y$ is contained in a \emph{smooth} curve, $T$, of degree $d$, we have $0 \to \oc (-d+c) \to \ic _Y(c) \to \ic _{Y,T}(c) \to 0$ and we see that, if $c \geq d$, $\ic _Y(c)$ is globally generated if and only if the line bundle $\lc = \ic _{Y,T}(c)$ on $T$ is globally generated. But there are gaps in the degrees of globally generated line bundles on a smooth plane curve of degree $d$ (it is classically known that no such bundle exists if $d \geq 3$ and $1 \leq \deg \lc \leq d-2$). A remarkable theorem due to Greco-Raciti and Coppens (\cite{GR}, \cite{Co} and Section \ref{S: statement}) gives the exact list of gaps.
 \par
This is another obstruction, at least if $Y$ lies on a \emph{smooth} curve, $T$, of low degree with respect with $c=c_1(F)$ (in this case $F$ tends to be not stable). The problem then is to have such a curve for \underline{every} vector bundle with fixed Chern classes and then, to treat the case where $T$ is not smooth. The first problem is solved in the \emph{necessarily unstable} range ($\Delta (F)=c_1^2-4c_2 > 0$) (see Section \ref{S: statement}). In the stable range there are no obstructions, this was already known to Le Potier (see \cite{LePotier}). For the second problem we use the following remark: if a line bundle $\oc _T(Z)$ on a smooth plane curve of degree $d$ is globally generated, then $Z$ satisfies Cayley-Bacharach for $d-3$. Working with the minimal section of $F$ we are able to have a similar statement even if $T$ is singular (see \ref{Z CBt-4}). Finally with a slight modification of Theorem 3.1 in \cite{GR} we are able to show the existence of gaps. 
\par
To state our result we need some notations. Let $c>0$ be an integer. Let's say that $(c,y)$ is effective if there exists a globally generated rank two vector bundle on $\Ptw$, $F$, with $c_1(F)=c, c_2(F)=y$. It is easy to see (cf Section \ref{S: gen}) that it must be $0 \leq y \leq c^2$ and that $(c,y)$ is effective if and only if $(c, c^2-y)$ is. So we may assume $y \leq c^2/2$. For every integer $t$, $2 \leq t \leq c/2$, let $G_t(0)=[c(t-1)+1,\,\,t(c-t)-1]$ (we use the convention that if $b<a$, then $[a,b]=\emptyset$). For every integer $t$, $4 \leq t \leq c/2$, denote by $t_0$ the integral part of $\sqrt{t-3}$, then for every integer $a$ such that $1 \leq a \leq t_0$ define $G_t(a) =[(t-1)(c-a)+a^2+1\, ,(t-1)(c-a+1)-1]$. Finally let 
$$G_t= \bigcup _{a=0}^{t_0}G_t(a)\,\,\, and\,\,\, G = \bigcup _{t=2}^{c/2}G_t\,\,.$$ Then we have:   

\begin{theorem}
\label{Thm}
Let $c > 0$ be an integer. There exists a globally generated rank two vector bundle on $\Ptw$ with Chern classes $c_1=c,\, c_2=y$ if and only if
one of the following occurs:
\begin{enumerate}
\item $y=0$ or $c-1 \leq y < c^2/4$ and $y \notin G$
\item $c^2/4 \leq y \leq 3c^2/4$
\item $3c^2/4 < y \leq c^2-c+1$ and $c^2-y \notin G$ or $y=c^2$.
\end{enumerate} 
\end{theorem}

\par
Although quite awful to state, this result is quite natural (see Section \ref{S: statement}). As a by-product we get (Section \ref{S: G}) all the possible ``bi-degrees'' for generically injective morphisms from $\Ptw$ to the Grassmannian $G(1,3)$ (or more generally to a Grassmannian of lines). To conclude let's mention that some partial results on this problem can be found in \cite{ELN}. 
\par
\emph{Acknowledgment:} I thank L. Gruson for drawing this problem to my attention and E. Mezzetti for pointing out a misprint in the statement of the theorem.

\section{General facts and a result of Le Potier for stable bundles.} \label{S: gen}

Let $F$ be a rank two globally generated vector bundle on $\Ptw$ with Chern classes $c_1(F)=:c$, $c_2(F)=:y$. Since the restriction $F_L$ to a line is globally generated, we get $c \geq 0$. A general section of $F$ yields:
$$0 \to \oc \to F \to \ic _Y(c) \to 0$$
where $Y \subset \Ptw$ is a smooth set of $y$ distinct points (cf \cite{Ha}, 1.4) or is empty. In the first case $y > 0$, in the second case $F \simeq \oc \oplus \oc (c)$ and $y=0$. In any case the Chern classes of a globally generated rank two vector bundle are positive.
\par
Also observe ($Y \neq \emptyset$) that $\ic _Y(c)$ is globally generated (in fact $F$ globally generated $\Leftrightarrow \ic _Y(c)$ is globally generated). This implies by Bertini's theorem that a general curve of degree $c$ containing $Y$ is smooth (hence irreducible).
\par
Since $rk(F)+\dim (\Ptw )=4$, $F$ can be generated by $V \subset H^0(F)$ with $\dim V =4$ and we get:
$$0 \to E^* \to V \otimes \oc \to F \to 0$$
It follows that $E$ is a rank two, globally generated vector bundle with Chern classes: $c_1(E)=c$, $c_2(E)=c^2-y$. We will say that $E$ is the \emph{G-dual bundle} of $F$. Since a globally generated rank two bundle has positive Chern classes we get: $0 \leq y \leq c^2$, $c \geq 0$.

\begin{definition}
We will say that $(c,y)$ is \emph{effective} if there exist a globally generated rank two vector bundle on $\Ptw$ with $c_1=c$ and $c_2=y$. A non effective $(c,y)$ will also be called a \emph{gap}.
\end{definition}

\begin{remark}
\label{rmk c^2/2}
By considering G-dual bundles we see that $(c,y)$ is effective if and only if $(c,c^2-y)$ is effective. Hence it is enough to consider the range $0 \leq y \leq c^2/2$.
\par
If $c=0$, then $F \simeq 2.\oc$ and $y=0$.
\par
If $y = c^2$ then $c_2(E)=0$, hence $E \simeq \oc \oplus \oc (c)$ and:
$$0 \to \oc (-c) \to 3.\oc \to F \to 0$$
Such bundles exists for any $c \geq 0$. If $y=0$, $F \simeq \oc \oplus \oc (c)$.
\end{remark}

\begin{definition}
\label{stableBdles}
If $F$ is a rank two vector bundles on $\Ptw$ we denote by $F_{norm}$ the unique twist of $F$ such that $-1 \leq c_1(F_{norm}) \leq 0$. The bundle $F$ is \emph{stable} if $h^0(F_{norm})=0$.
\end{definition}

\par
By a result of Schwarzenberger, if $F$ is stable with $c_1(F)=c, c_2(F)=y$, then $\Delta (F) := c^2-4y < 0$ (and $\Delta (F) \neq -4$). Moreover there exist a stable rank two vector bundle with Chern classes $(c,y)$ if and only if $\Delta := c^2 -4y < 0$, $\Delta \neq -4$.
\par 
Concerning stable bundles we have the following result of Le Potier \cite{LePotier}:

\begin{proposition}[Le Potier]
\label{Le Potier}
Let $\mc (c_1,c_2)$ denote the moduli space of stable rank two bundles with Chern classes $c_1,c_2$ on $\Ptw$. There exists a non empty open subset of $\mc (c_1,c_2)$ corresponding to globally generated bundles if and only if one of the following holds:\\
(1) $c_1 > 0$ and $\chi (c_1,c_2) \geq 4$ ($\chi (c_1,c_2)= 2 +\frac{c_1(c_1+3)}{2}-c_2$)\\
(2) $(c_1,c_2)=(1,1)$ or $(2,4)$.
\end{proposition}

Using this proposition we get:

\begin{corollary}
\label{CorLP}
If $c >0$ and
$$\frac{c^2}{4} \leq y \leq \frac{3c^2}{4}$$
then $(c,y)$ is effective.
\end{corollary}

\begin{proof} The existence condition ($\Delta < 0, \Delta \neq -4$) translates as: $y > c^2/4$, $y \neq c^2/4 +1$. Condition (1) of \ref{Le Potier} gives: $\frac{c(c+3)}{2}-2 \geq y$, hence if $\frac{c(c+3)}{2}-2 \geq y > \frac{c^2}{4}$ and $y \neq \frac{c^2}{4}+1$, $(c,y)$ is effective.
\par
Let's show that $(c,\frac{c^2}{4})$ is effective for every $c \geq 2$ ($c$ even). Consider:
$$0 \to \oc \to F \to \ic _Y(2) \to 0$$
where $Y$ is one point. Then $F$ is globally generated with Chern classes $(2,1)$. For every $m \geq 0$, $F(m)$ is globally generated with $c_1^2 =4c_2$.
\par
In the same way let's show that $(c, \frac{c^2}{4}+1)$ is effective for every $c \geq 2$ ($c$ even). This time consider:
$$0 \to \oc \to F \to \ic _Y(2) \to 0$$
where $Y$ is a set of two points; $F$ is globally generated with Chern classes $(2,2)$. For every $m \geq 0$, $F(m)$ is globally generated with the desired Chern classes.
\par
We conclude that if $\frac{c(c+3)}{2}-2 \geq y \geq \frac{c^2}{4}$, then $(c,y)$ is effective. By \emph{duality}, $(c,y)$ is effective if $\frac{3c^2}{4} \geq y \geq \frac{c(c-3)}{2}+2$. Putting every thing together we get the result.
\end{proof}

\begin{remark}
\label{y<c^2/4} 
Since $3c^2/4 > c^2/2$, we may, by \emph{duality}, concentrate on the range $y < c^2/4$, i.e. on not stable bundles with $\Delta > 0$, that's what we are going to do in the next section.
\end{remark}


\section{Cayley-Bacharach.}\label{S:CB}

\begin{definition}
\label{CB def}
Let $Y \subset \Ptw$ be a locally complete intersection (l.c.i.) zero-dimensional subscheme. Let $n \geq 1$ be an integer. We say that $Y$ satisfies Cayley-Bacharach for curves of degree $n$ ($CB(n)$), if any curve of degree $n$ containing a subscheme $Y' \subset Y$ of colength one (i.e. of degree $\deg Y -1$), contains $Y$.
\end{definition}

\begin{remark} Since $Y$ is l.c.i. for any $p \in Supp(Y)$ there exists a unique subscheme $Y' \subset Y$ of colength one (locally) linked to $p$ in $Y$. So Def. \ref{CB def} makes sense even if $Y$ is non reduced.
\end{remark}

Let's recall the following (\cite{GH}):

\begin{proposition}
\label{CB}
Let $Y \subset \Ptw$ be a zero-dimensional l.c.i. subscheme. There exists an exact sequence:
$$0 \to \oc \to F \to \ic _Y(c) \to 0$$
with $F$ a rank two vector bundle if and only if $Y$ satisfies $CB(c-3)$.
\end{proposition} 

See \cite{GH} (under the assumption that $Y$ is reduced) and \cite{Ca} for the general case. The proposition gives conditions on the Chern classes of bundles having a section, in our case:

\begin{lemma}
\label{Bacharach1}
Let $F$ be a globally generated rank two vector bundle on $\Ptw$ with $c_1(F)=c$, $c_2(F)=y$, then:
$$c-1 \leq y \leq c^2-c+1\,\,\,\,or\,\,\,\, y=c^2\,\,\,or\,\,\,\,\,y=0$$
\end{lemma}

\begin{proof} Since $F$ is globally generated a general section vanishes in codimension two or doesn't vanish at all. In the second case $F \simeq 2.\oc$ and $y=0$. Let's assume, from now on, that a general section vanishes in codimension two. We have an exact sequence:
$$0 \to \oc \to F \to \ic _Y(c) \to 0$$
where $Y$ is a zero-dimensional subscheme (we may assume $Y$ smooth) which satisfies Cayley-Bacharach condition for $c-3$.
\par
If $c-3 \geq y-1$, $\forall p \in Y$ there exists a curve of degree $c-3$ containing $Y_p :=Y \setminus \{p\}$ and not containing $Y$ (consider a suitable union of lines). Since $Y$ must satisfy the Cayley-Bacharach condition, it must be $y \geq c-1$.
\par
Let $F$ be a globally generated rank two vector bundle with $c_1(F)=c$, $c_2(F)=y$. Consider the \emph{G-dual} bundle:
$$0 \to F^* \to 4.\oc \to E \to 0$$
then $E$ is a rank two, globally generated, vector bundle with $c_1(E)=c$, $c_2(E)=c^2-y$. By the previous part: $c_2(E)=0$ (i.e. $y=c^2$) or $c^2-y = c_2(E) \geq c_1(E)-1 = c-1$. So $c^2-c+1 \geq y$.
\end{proof}

\begin{remark}
\label{c<4}
It is easy to check that for $1 \leq c \leq 3$, every value of $y$, $c-1 \leq y \leq c^2-c+1$ is effective (take $Y \subset \Ptw$ of maximal rank with $c-1 \leq y \leq c^2/2$ and use Castelnuovo-Mumford's lemma to show that $\ic _Y(c)$ is globally generated). In fact gaps occur only for $c \geq 6$. In the sequel we will assume that $c \geq 4$.
\end{remark}

\section{The statement.} \label{S: statement}

\par
From now on we may restrict our attention to the range: $c-1 \leq y < c^2/4$ (\ref{y<c^2/4}, \ref{Bacharach1}) for $c \geq 4$ (\ref{c<4}). In this range $\Delta (F) = c^2-4y > 0$, hence $F$ is \emph{necessarily unstable} (i.e. not semi-stable). In particular, if $c$ is even: $h^0(F(-\frac{c}{2}))=h^0(\ic _Y(\frac{c}{2})) \neq 0$ (resp. $h^0(F(-\frac{(c+1)}{2}))=h^0(\ic _Y(\frac{c-1}{2}))\neq 0$, if $c$ is odd). So $Y$ is forced to lie on a curve of relatively low degree. In fact something more precise can be said, for this we need the following elementary remark: 

\begin{lemma}[The trick]
\label{trick}
Let $F$ be a rank two vector bundle on $\Ptw$ with $h^0(F) \neq 0$. If $c_2(F) < 0$, then $h^0(F(-1)) \neq 0$.
\end{lemma}

\begin{proof} A non-zero section of $F$ cannot vanish in codimension two (we would have $c_2 >0$), nor can the section be nowhere non-zero ($F$ would split as $F \simeq \oc \oplus \oc (c)$, hence $c_2(F)=0)$. It follows that any section vanishes along a divisor. By dividing by the equation of this divisor we get $h^0(F(-1)) \neq 0$.
\end{proof}

Actually this works also on $\Pn$, $n\geq 2$.
\par
For $2 \leq t \leq c/2$ ($c \geq 4$) we define:
$$\overline{A}_t := [(t-1)(c-t+1),\,\,t(c-t)]  = [(t-1)c-(t-1)^2,\,\,\,(t-1)c-(t^2-c)]$$

The ranges $\overline A_t$ cover $[c-1,\,\,\frac{c^2}{4}[$, the interval we are interested in. From our point of view we may concentrate on the interior points of $\overline A_t$. Indeed if $y = ab$, with $a+b=c$, we may take $F \simeq \oc (a)\oplus \oc (b)$. So we define:
$$A_t = ](t-1)(c-t+1),\,\,\,t(c-t)[,\,\,\,2 \leq t \leq c/2$$

\begin{lemma}
\label{forced t-1}
If $y \in A_t$, and if $Y$ is the zero-locus of a section of $F$, a rank two vector bundle with Chern classes $(c,y)$, then $h^0(\ic _Y(t-1))\neq 0$.
\end{lemma}

\begin{proof} We have an exact sequence $0 \to \oc \to F \to \ic _Y(c) \to 0$. Now $c_2(F(-(c-t))=(-c+t)t+y$. By our assumptions, $y < t(c-t)$, hence $c_2(F(-c+t)) < 0$. Looking at the graph of $c_2(F(x)) = x^2 + cx +y$, we see that $c_2(F(x)) < 0$ for $(-c-\sqrt{\Delta (F)})/2 < x \leq -c/2$. Since $c_2(F(-c+t)) < 0$, $-c+t < -c/2$ and $h^0(F(-c/2))\neq 0$, by induction, using Lemma \ref{trick}, we conclude that $h^0(F(-c+t-1))=h^0(\ic _Y(t-1)) \neq 0$.
\end{proof}

\par
So if $y \in A_t$, $Y$ is forced to lie on a degree $(t-1)$ curve (but not on a curve of degree $t-2$). If general principles are respected we may think that if $y \in A_t$, $Y \subset T$, where $T$ is a \emph{smooth} curve of degree $t-1$ and that $h^0(\ic _Y(t-2))=0$. If this is the case we have an exact sequence:
$$0 \to \oc (-t+1) \to \ic _Y \to \ic _{Y,T} \to 0$$
twisting by $\oc _T(c)$:
$$0 \to \oc (c-t+1) \to \ic _Y(c) \to \oc _T(c-Y) \to 0$$
Since $c-t+1 > 0$ (because $c \geq 2t$), we see that: $\ic _Y(c)$ is globally generated if and only if $\oc _T(c-Y)$ is globally generated. The line bundle $\lc = \oc _T(c-Y)$ has degree $l:= c(t-1)-y$. So the question is: for which $l$ does there exists a degree $l$ line bundle on $T$ generated by global sections? This is, by its own, a quite natural problem which, strangely enough, has been solved only recently (\cite {GR}, \cite{Co}). First a definition:

\begin{definition}
\label{LS(t-1)}
Let $C$ be a smooth irreducible curve. The L\H{u}roth semi-group of $C$, $LS(C)$, is the semi-group of nonnegative integers which are degrees of rational functions on $C$. In other words: $LS(C) = \{n \in \bN \mid \exists \lc$, of degree $n$, such that $\lc$ is globally generated $\}$.
\end{definition}

Then we have:

\begin{theorem}[Greco-Raciti-Coppens]
If $C$ is a smooth plane curve of degree $d \geq 3$, then 
$$LS(C) = LS(d) := \bN \setminus \bigcup _{a=1}^{n_0} [(a-1)d+1,\,\,a(d-a)-1]$$ 
where $n_0$ is the integral part of $\sqrt{d-2}$.
\end{theorem}

\par
Of course $LS(1)=LS(2)=\bN$. We observe that $LS(C)$ doesn't depend on $C$ but only on its degree.
\par
Going back to our problem we see that if $c(t-1)-y \notin LS(t-1)$, then  $\lc = \oc _T(c-Y)$ can't be globally generated and the same happens to $\ic _Y(c)$. In conclusion if $c(t-1)-y \in \displaystyle \bigcup _{a=1}^{n_0} [(a-1)(t-1)+1,\,\,\,a(t-1-a)-1]$, or if $c(t-1)-y < 0$, under our assumptions, $(c,y)$ is not effective. The assumption is that the \emph{unique} curve of degree $t-1$ containing $Y$ is \emph{smooth}. (Observe that $\deg \oc _T(t-1-Y)<0$, hence $h^0(\ic _Y(t-1))=1$.)
\par
Our theorem says that general principles are indeed respected. In order to have a more manageable statement let's introduce some notations:

\begin{definition}
\label{non-adm}
Fix an integer $c \geq 4$. An integer $y \in A_t$ for some $2 \leq t \leq c/2$, will be said to be \emph{admissible} if $c(t-1)-y \in LS(t-1)$. If $c(t-1)-y \notin LS(t-1)$, $y$ will be said to be \emph{non-admissible}.
\par
Observe that $y \in A_t$ is non-admissible if and only if: $y \in G_t(0) = [c(t-1)+1\,,\,t(c-t)-1]$ (this corresponds to $c(t-1)-y < 0$), or $y \in G_t(a) = [(t-1)(c-a)+a^2+1\,,\,(t-1)(c-a+1)-1]$ for some $a \geq 1$ such that $a^2+2 \leq t-1$ (i.e. $a \leq t_0$).
\end{definition}

In order to prove Theorem \ref{Thm} it remains to show:

\begin{theorem}
\label{Thm2}
For any $c\geq 4$ and for any $y \in A_t$ for some $2 \leq t \leq c/2$, $(c,y)$ is effective if and only if $y$ is admissible.
\end{theorem}

\par
The proof splits into two parts:
\begin{enumerate}
\item {\it (Gaps)} If $c(t-1)-y \notin LS(t-1)$, one has to prove that $(c,y)$ is not effective. This is clear if $Y$ lies on a smooth curve, $T$, of degree $t-1$, but there is no reason for this to be true and the problem is when $T$ is singular.
\item {\it (Existence)} If $c(t-1)-y \in LS(t-1)$, one knows that there exists $\lc$ globally generated, of degree $c(t-1)-y$ on a smooth curve, $T$, of degree $t-1$. The problem is to find such an $\lc$ such that $\mc := \oc _T(c)\otimes \lc ^*$ has a section vanishing along a $Y$ satisfying the Cayley-Bacharach condition for $(c-3)$.
\end{enumerate}

\section{The proof (gaps).}

In this section we fix an integer $c \geq 4$ and prove that non-admissible $y \in A_t$, $2 \leq t \leq c/2$ are gaps. For this we will assume that such a $y$ is effective and will derive a contradiction. From \ref{forced t-1} we know that $h^0(\ic _Y(t-1))\neq 0$. The first task is to show that under our assumption ($y$ not-admissible), $h^0(\ic _Y(t-2))=0$ (see \ref{Yt-1=1 si gap}); this will imply that $F(-c+t-1)$ has a section vanishing in codimension two.
\par
To begin with let's observe that non-admissible $y \in A_t$ may occur only when $t$ is small with respect to $c$.

\begin{lemma}
\label{NoGap-tGrand}
Assume $c \geq 4$. If $t > \frac{2\sqrt 3}{3}\sqrt{c-2}$, then every $y \in A_t$ is admissible.
\end{lemma}

\begin{proof}
Recall (see \ref{non-adm}) that $y \in A_t$, $2 \leq t \leq c/2$, is non admissible if and only if $y \in G_t(a)$ for some $a, 0 \leq a \leq t_0$.
\par
We have $G_t(0) \neq \emptyset \Leftrightarrow t(c-t)-1 \geq c(t-1)+1 \Leftrightarrow t \leq \sqrt{c-2}$.
\par
For $a \geq 1$, $G_t(a) \cap A_t \neq \emptyset \Rightarrow (t-1)(c-a)+a^2+1 < t(c-t)$. This is equivalent to: $a^2-at+t^2-c+a+1 < 0\,\,(*)$. The discriminant of this equation in $a$ is $\Delta = -3t^2+4(c-a-1)$ and we must have $\Delta \geq 0$, i.e. $\frac{2\sqrt 3}{3}\sqrt{c-2}\geq t$.
\end{proof}

\par
Let's get rid of the $y's$ in $G_t(0)$:

\begin{lemma}
\label{c(t-1)-y<0}
If $y \in A_t$ is non-admissible and effective, then $y \in G_t(a)$ for some $a$, $1 \leq a \leq \sqrt{t-3}$.
\end{lemma}

\begin{proof} We have to show that if $c(t-1)-y<0$ and $y \in A_t$, then $y$ is not effective. By \ref{forced t-1} $h^0(\ic _Y(t-1))\neq 0$. If $y$ is effective then $\ic _Y(c)$ is globally generated and $Y$ is contained in a complete intersection of type $(t-1, c)$, hence $\deg Y=y \leq c(t-1)$: contradiction.
\end{proof}

Now we show that if $y$ is non-admissible and effective, then $h^0(\ic _Y(t-1))=1$:

\begin{lemma}
\label{Yt-1=1 si gap}
Let $c \geq 4$ and assume $y \in A_t$ for some $t$, $2 \leq t \leq c/2$. Assume furthermore that $y$ is non-admissible and effective i.e.:
$$y = (t-1)(c-a)+\aG, \,\,\,a^2+1 \leq \aG \leq t-2$$
for a given $a$ such that $t-1 \geq a^2 +2$. Under these assumptions, $h^0(\ic _Y(t-1))=1$.
\end{lemma}

\begin{proof}
If $h^0(\ic _Y(t-2))\neq 0$, then $y \leq c(t-2)$ (the general $F_c \in H^0(\ic _Y(c))$ is integral since $\ic _Y(c)$ is globally generated. Moreover $t-1 < c$ so $F_c \neq T$). It follows that:
$$y = (t-1)(c-a)+\aG \leq c(t-2)=c(t-1)-c$$
This yields $a(t-1) \geq c +\aG$. We have $c+\aG \geq c+a^2+1$, hence:
$$0 \geq a^2-a(t-1)+c+1\,\,\,\,(*)$$
The discriminant of $(*)$ (viewed as an equation in $a$) is: $\Delta = (t-1)^2-4(c+1)$. If $\Delta  < 0$, $(*)$ is never satisfied and $h^0(\ic _Y(t-2))=0$. Now $\Delta < 0 \Leftrightarrow (t-1)^2 < 4(c+1)$. In our context $\Delta < 0 \Leftrightarrow t < 1+2\sqrt{c+1}$. In conclusion if $t < 1+2\sqrt{c+1}$ and if $y$ is non-admissible, then $h^0(\ic _Y(t-2))=0$.  
\par
Now by \ref{NoGap-tGrand} if $y$ is non-admissible, we have: $t \leq \frac{2\sqrt 3}{3}\sqrt{c-2}$. Since $\frac{2\sqrt 3}{3}\sqrt{c-2} < 1+2\sqrt{c+1}$, $\forall c > 0$, we are done.
\par
Since $h^0(\ic _Y(t-1))\neq 0$, $F(-c+t-1)$ has a non-zero section, since $h^0(\ic _Y(t-2))=0$ the section vanishes in codimension two. Hence we have:
$$0 \to \oc \to F(-c+t-1) \to \ic _W(-c+2t-2) \to 0$$
where $\deg W =y - (t-1)(c-t+1)$. Since $-c+2t-2 < 0$ (because $c \geq 2t$), we get $h^0(F(-c+t-1))=1=h^0(\ic _Y(t-1))$.
\end{proof}

\begin{notations}
\label{notation gap}
Let $F$ be a globally generated rank two vector bundle with Chern classes $(c,y)$. A section $s \in H^0(F)$ defines $Y_s = (s)_0$. If $y \in A_t$, $h^0(\ic _{Y_s}(t-1))\neq 0$, moreover if $y$ is non-admissible $h^0(\ic _{Y_s}(t-1))=1$ and there is a unique $T_s \in H^0(\ic _{Y_s}(t-1))$. It follows that $F(-c+t-1)$ has a unique section (hence vanishing in codimension two): $0 \to \oc \stackrel{u}{\to} F(-c+t-1) \to \ic _W(-c+2t-2) \to 0$.
\end{notations}

\begin{lemma}
\label{T reduite}
If $y \in A_t$ is non-admissible and effective, with notations as in \ref{notation gap}:
\begin{enumerate}
\item $Y_s$ and $W$ are bilinked on $T_s$
\item The curves $T_s$ are precisely the elements of $H^0(\ic _W(t-1))$
\item $\ic _W(t-1)$ is globally generated, in particular for $s \in H^0(F)$ general, $T_s$ is reduced.
\end{enumerate}
\end{lemma}

\begin{proof}\quad \\
(1) (2) We have a commutative diagram:
$$\begin{array}{ccccccccc}

 		&					& 								& 										& 0									& 			& 0										& 			&  \\
 		&					& 								& 										& \downarrow				& 			& \downarrow					& 			&  \\
 		&					& 								& 										& \oc								& 	=		& \oc									& 			&  \\
 		&					& 								& 										& \downarrow u			& 			& \downarrow T_s				& 			&  \\
0		&		\to		& \oc (-c+t-1)		& \stackrel{s}{\to}		& F(-c+t-1)					& \to		& \ic _Y(t-1)					& \to		& 0 \\

 		&					& ||							& 										& \downarrow				& 			& \downarrow					& 			&  \\

0		&		\to		& \oc (-c+t-1)		& \stackrel{s}{\to}		& \ic _W(-c+2t-2)		& \to		& \ic _{Y,T_s}(t-1)		& \to		& 0 \\
 		&					& 								& 										& \downarrow				& 			& \downarrow					& 			&  \\
 		&					& 								& 										& 0									& 			& 0										& 			&  \end{array}$$
 		
\par
This diagram is obtained as follows: the section $T_s$ lifts to, $u$, the unique section of $F(-c+t-1)$, hence $Coker(u) \simeq \ic _W(-c+2t-2)$, then take the first horizontal line corresponding to $s$ and complete the full diagram.

We see that $s$ corresponds to an element of $H^0(\ic _W(t-1))$ and the quotient $\oc (-t+1) \stackrel{s}{\to}\ic _W$ has support on $T_s$ and is isomorphic to $\ic _{Y,T_s}(-t+c+1)$, hence $\ic _{W,T_s}(-h) \simeq \ic _{Y,T_s}$ where $h=c-t+1$. This shows that $W$ and $Y_s$ are bilinked on $T_s$. Indeed by composing with the inclusion $\ic _{Y,T_s} \to \oc _{T_s}$ we get an injective morphism $\varphi :\ic _{W,T_s}(-h) \to \oc _{T_s}$, now take a curve $C \in H^0(\ic _W(k))$, without any irreducible component in common with $T_s$ and take $C'$ such that $\varphi (C)=C'$ in $H^0_*(\oc _{T_s})$, then $W$ is bilinked to $Y$ by $C\cap T_s$ and $C'\cap T_s$.

This can be seen in another way: take $C'\in H^0(\ic _Y(k))$ with no irreducible common component with $T_s$, the complete intersection $T_s \cap C'$ links $Y$ to a subscheme $Z$. By mapping cone:
$$0 \to F^*(h-k) \to \oc (-t+1) \oplus \oc (-k)\oplus \oc (-(k-h)) \stackrel{(T_s,C',C)}{\to} \ic _Z \to 0$$
Observe that $C$ and $T_s$ do not share any common component. Indeed if $T_s=AT'$ and $C=A\tilde C$, then $(C'\cap A)\subset Z$ (as schemes), because $Z$ is the schematic intersection of $T_s, C'$ and $C$. This is impossible because $C'\cap A$ contains points of $Y$ (otherwise $Y \subset T'$ but $h^0(\ic _Y(t-2))=0$). The complete intersection $T_s\cap C$ links $Z$ to a subscheme $W'$ and by mapping cone, we get that $W'$ is a section of $F(-h)$. By uniqueness it follows that $W'=W$. The same argument starting from $W$ and $T\in H^0(\ic _W(t-1))$, instead of $Y$ and $T_s$, works even better and shows that $W$ is bilinked on $T$ to a section $Y_s$ of $F$. In conclusion the curves $T_s$ are given by $s \wedge u$, where $u$ is the unique section of $F(-h)$ and where $s\in H^0(F)$ vanishes in codimension two.

(3) The exact sequence $0 \to \oc (c-t+1) \to F \to \ic _W(t-1) \to 0$ shows that $\ic _W(t-1)$ is globally generated, hence the general element in $H^0(\ic _W(t-1))$ is reduced.
\end{proof}

\par
Since $W$ could well be non reduced with embedding dimension two, concerning $T$, this is the best we can hope. However, \underline{and this is the point}, we may reverse the construction and start from $W$.

\begin{lemma}
\label{link W}
Let $W \subset \Ptw$ be a zero-dimensional, locally complete intersection (l.c.i.) subscheme. Assume $\ic _W(n)$ is globally generated, then if $T, T' \in H^0(\ic _W(n))$ are sufficiently general, the complete intersection $T \cap T'$ links $W$ to a smooth subscheme $Z$ such that $W \cap Z=\emptyset$.
\end{lemma}

\begin{proof} If $p \in Supp(W)$, denote by $W_p$ the subscheme of $W$ supported at $p$. Since $W$ is l.c.i, $\ic _{W,p}=(f,g) \subset \oc _p$. By assumption the map $H^0(\ic _W(n))\otimes \oc _p \stackrel{ev}{\to} \ic _{W,p}$ which takes $T \in H^0(\ic _W(n))$ to its germ, $T_p$, at $p$, is surjective. Hence there exists $T$ such that $T_p =f$ (resp. $T'$ such that $T'_p =g$). It follows that in a neighbourhood of $p$: $T \cap T' = W_p$. If $G$ is the Grassmannian of lines of $H^0(\ic _W(n))$ for $\lag T,T' \rag \in G$ the property $T\cap T' =W_p$ (in a neighbourhood of $p$) is open (it means that the local degree at $p$ of $T\cap T'$ is minimum). We conclude that there exists a dense open subset, $U_p \subset G$, such that for $\lag T,T' \rag \in U_p$, $T \cap T' = W_p$ (locally at $p$). If $Supp(W)=\{p_1,...,p_r\}$ there exists a dense open subset $U \subset U_1\cap ...\cap U_r$ such that if $\lag T,T' \rag \in U$, then $T\cap T'$ links $W$ to $Z$ and $W \cap Z = \emptyset$.
\par
By Bertini's theorem the general curve $T \in H^0(\ic _W(n))$ is smooth out of $W$. If $C \subset T$ is an irreducible component, the curves of $H^0(\ic _W(n))$ cut on $C$, residually to $W\cap C$, a base point free linear system. By the previous part the general member, $Z_C$, of this linear system doesn't meet $Sing(C)$ (because $Z_C \cap W = \emptyset$), it follows, by Bertini's theorem, that $Z_C$ is smooth. So for general $T,T' \in H^0(\ic _W(n))$, $T\cap T'$ links $W$ to a smooth subscheme, $Z$, such that $W \cap Z=\emptyset$.
\end{proof}

\begin{corollary}
\label{linkage WZ}
Let $y \in A_t$ be non-admissible. If $y$ is effective, with notations as in \ref{notation gap}, if $T,T' \in H^0(\ic _W(t-1))$ are sufficiently general, then $T \cap T'$ links $W$ to a smooth subscheme, $Z$, such that $W \cap Z = \emptyset$. Furthermore $\ic _Z(c)$ is globally generated and if $S_c \in H^0(\ic _Z(c))$ is sufficiently general, then $T \cap S_c$ links $Z$ to a smooth subscheme $Y$, where $Y$ is the zero locus of a section of $F$ and where $Z \cap Y=\emptyset$.
\end{corollary}

\begin{proof} The first statement follows from \ref{link W}. From the exact sequence $$0 \to \oc (c-2t+2) \to F(-t+1) \to \ic _W \to 0$$ we get by mapping cone: 
$$0 \to F^*(-t+1) \to \oc (-c) \oplus 2.\oc (-t+1) \to \ic _Z \to 0\,\,\,\,(*)$$
which shows that $\ic _Z(c)$ is globally generated. Since $Z$ is smooth and contained in the smooth locus of $T$ and since $\ic _Z(c)$ is globally generated, if $C$ is an irreducible component of $T$, the curves of $H^0(\ic _Z(c))$ cut on $C$, residually to $C\cap Z$, a base point free linear system. In particular the general member, $D$, of this linear system doesn't meet $Sing(C)$. By Bertini's theorem we may assume $D$ smooth. It follows that if $S_c \in H^0(\ic _Z(c))$ is sufficiently general, $S_c \cap T$ links $Z$ to a smooth $Y$ such that $Z \cap Y=\emptyset$. By mapping cone, we see from $(*)$ that $Y$ is the zero-locus of a section of $F$.
\end{proof}

The previous lemmas will allow us to apply the following (classical, I think) result:

\begin{lemma}
\label{CB residuel}
Let $Y, Z \subset \Ptw$ be two zero-dimensional subschemes linked by a complete intersection, $X$, of type $(a,b)$. Assume:
\begin{enumerate}
\item $Y \cap Z = \emptyset$
\item $\ic _Y(a)$ globally generated.
\end{enumerate}
Then $Z$ satisfies Cayley-Bacharach for $(b-3)$.
\end{lemma}

\begin{proof} Notice that $Z$ and $Y$ are l.c.i. Now let $P$ be a curve of degree $b-3$ containing $Z' \subset Z$ of colength one. We have to show that $P$ contains $Z$. Since $\ic _Y(a)$ is globally generated and since $Y \cap Z =\emptyset$, there exists $F \in H^0(\ic _Y(a))$ not passing through $p$. Now $PF$ is a degree $a+b-3$ curve containing $X \setminus \{p\}$. Since complete intersections $(a,b)$ verify Cayley-Bacharach for $a+b-3$ (the bundle $\oc (a)\oplus \oc (b)$ exists!), $PF$ passes through $p$. This implies that $P$ contains $Z$.
\end{proof}

Gathering everything together:

\begin{corollary}
\label{Z CBt-4}
Let $y \in A_t$ be non-admissible. If $y$ is effective, then there exists a smooth zero-dimensional subscheme $Z$ such that:
\begin{enumerate}
\item $Z$ lies on a pencil $\lag T,T'\rag$ of curves of degree $t-1$, the base locus of this pencil is zero-dimensional.
\item $\deg Z =c(t-1)-y$
\item $Z$ satisfies Cayley-Bacharach for $t-4$
\end{enumerate}
\end{corollary}

\begin{proof} By \ref{linkage WZ} there is a $Y$ zero-locus of a section of $F$ which is linked by a complete intersection of type $(c,t-1)$ to a $Z$ such that $Y \cap Z = \emptyset$. Since $\ic _Y(c)$ is globally generated, by \ref{CB residuel}, $Z$ satisfies Cayley-Bacharach for $t-4$.
\end{proof}

Now we conclude with:

\begin{proposition}
\label{Exist Gaps}
Let $Z \subset \Ptw$ be a smooth zero-dimensional subscheme contained in a curve of degree $d$. Let $a \geq 1$ be an integer such that $d \geq a^2+2$. Assume $h^0(\ic _Z(a-1))=0$. If $(a-1)d+1 \leq \deg Z \leq a(d-a)-1$, then $Z$ doesn't verify Cayley-Bacharach for $d-3$.
\end{proposition}

\begin{remark} This proposition is Theorem 3.1 in \cite{GR} with a slight modification: we make no assumption on the degree $d$ curve (which can be singular, even non reduced), but we assume $h^0(\ic _Z(a-1))=0$ (which follows from Bezout if the degree $d$ curve is integral).
\par
Since this proposition is a key point, and for convenience of the reader, we will prove it. We insist on the fact that the proof given is essentially the proof of Theorem 3.1 in \cite{GR}.
\end{remark}

\begin{notations}
We recall that if $Z \subset \Ptw$, the numerical character of $Z$, $\chi = (n_0,...,n_{\sG -1})$ is a sequence of integers which encodes the Hilbert function of $Z$ (see \cite{GP}):
\begin{enumerate}
\item $n_0 \geq ... \geq n_{\sG -1} \geq \sG$ where $\sG$ is the minimal degree of a curve containing $Z$
\item $h^1(\ic _Z(n)) = \displaystyle \sum _{i=0}^{\sG -1}[n_i-n-1]_+ - [i-n-1]_+$  ($[x]_+=max\{0,x\}$).
\item In particular $\deg Z = \displaystyle \sum _{i=0}^{\sG -1}(n_i-i)$.
\end{enumerate}
The numerical character is said to be \emph{connected} if $n_i \leq n_{i+1}+1$, for all $0 \leq i < \sG -1$.
For those more comfortable with the Hilbert function, $H(Z,-)$ and its first difference function, $\Delta (Z,i)=H(Z,i)-H(Z,i-1)$, we recall that $\Delta (i)=i+1$ for $i < \sG$ while $\Delta (i) = \# \{l \mid n_l \geq i+1\}$. It follows that the condition $n_{r-1}>n_r+1$ is equivalent to $\Delta (n_r+1)=\Delta (n_r)$. Also recall that for $0 \leq i < \sG$, $n_i=min\,\{t \geq i \mid \Delta (t) \leq i\}$.
\end{notations}

\begin{lemma}
\label{trou}
Let $Z \subset \Ptw$ be a smooth zero-dimensional subscheme. Let $\chi = (n_0,...,n_{\sG -1})$ be the numerical character of $Z$. If $n_{r-1}>n_r+1$, then $Z$ doesn't verify Cayley-Bacharach for every $i \geq n_r-1$.
\end{lemma}

\begin{proof} It is enough to show that $Z$ doesn't verify $CB(n_r-1)$. By \cite{EP} there exists a curve, $R$, of degree $r$ such that $R \cap Z =E'$ where $\chi (E')=(n_0,...,n_{r-1})$. Moreover if $E''$ is the residual of $Z$ with respect to the divisor $R$, $\chi (E'')=(m_0,...,m_{\sG -1-r})$, with $m_i=n_{r+i}-r$. It follows that $h^1(\ic _{E''}(n_r-r-1))=0$. This implies that given $X \subset E''$ of colength one, there exists a curve, $P$, of degree $n_r-r-1$ passing through $X$ but not containing $E''$. The curve $RP$ has degree $n_r-1$, passes through $Z':=E'\cup X$ but doesn't contain $Z$ (because $R\cap Z=E'$).
\end{proof}

\begin{proof}[Proof of Proposition \ref{Exist Gaps}]\quad \\ Observe that the assumptions imply $d \geq 3$, moreover if $d=3$, then $a = \deg Z =1$ and the statement is clear; so we may assume $d \geq 4$.
\par
Assume to the contrary that $Z$ satisfies $CB(d-3)$. This implies $h^1(\ic _Z(d-3))\neq 0$. If $a=1$, then $\deg Z \leq d-2$ and necessarily $h^1(\ic _Z(d-3))=0$, so we may assume $a \geq 2$. Now if $h^1(\ic _Z(d-3))\neq 0$, then $n_0 \geq d-1$, where $\chi (Z)=(n_0,...,n_{\sG -1})$ is the numerical character of $Z$. Since $\sG \geq a$, $n_{a-1} \in \chi (Z)$.
\par
We claim that $n_{a-1}<d-2$. Indeed otherwise $n_0 \geq d-1$ and $n_0 \geq ... \geq n_{a-1} \geq d-2$ implies 
$$\deg Z = \sum _{i=0}^{\sG -1} (n_i-i) \geq \sum _{i=0}^{a -1} (n_i-i) \geq 1+\sum _{i=0}^{a -1} (d-2-i)=1+a(d-2)-\frac{a(a-1)}{2}$$
If $a \geq 1$, then $1+a(d-2)-\frac{a(a-1)}{2} > a(d-a)-1 \geq \deg Z$: contradiction. 
\par
Let's show that $n_{a-1}\geq d-a$. Assume to the contrary $n_{a-1}<d-a$. Then there exists $k$, $1 \leq k \leq a-1$ such that $n_k \leq d-2$ and $n_{k-1}\geq d-1$ (indeed $n_0 \geq d-1$ and $n_{a-1}<d-a \leq d-2$). If $n_{k-1}\geq n_k\geq ... \geq n_{a-1}$ is connected, then $n_{k-1} < d-a+r$ where $a=k+r$. Hence $d-a+r > n_{k-1}\geq d-1$, which implies $r \geq a$ which is impossible since $k\geq 1$. It follows that there is a gap in $n_{k-1}\geq n_k\geq ... \geq n_{a-1}$, i.e. there exists $r$, $k \leq r \leq a-1$, such that $n_{r-1}>n_r+1$. Since $d-2 \geq n_k \geq n_r$, we conclude by \ref{trou} that $Z$ doesn't satisfy $CB(d-3)$: contradiction.  
\par
So far we have $d-a \leq n_{a-1}<d-2$ and $n_0 \geq d-1$. Set $n_{a-1}=d-a+r$ ($r \geq 0$). We claim that there exists $k$ such that $n_k \geq d-1$ and $n_k \geq ... \geq n_{a-1}=d-a+r$ is connected. Since $n_0\geq d-1$, this follows from \ref{trou}, otherwise $Z$ doesn't verify $CB(d-3)$.
\par
We have $\chi (Z)=(n_0,...,n_k,...,n_{a-1},...,n_{\sG -1})$ with $n_k \geq d-1$, $n_{a-1}=d-a+r$. Since $(n_k,..., n_{a-1})$ is connected and $n_k \geq d-1$, we have $n_i \geq d-1+k-i$ for $k \leq i \leq a-1$. Since $n_{a-1}=d-a+r \geq d-1+k-(a-1)$, we get $r \geq k$. It follows that:
$$\deg Z = \sum _{i=0}^{\sG -1}(n_i-i) = \sum _{i=0}^{k-1}(n_i-i)+\sum _{i=k}^{a-1}(n_i-i)+\sum _{i \geq a}(n_i-i)$$
$$\geq \sum _{i=0}^{k-1}(d-1-i)+\sum _{i=k}^{a-1}(d-1-2i+k)+\sum _{i \geq a}(n_i-i)$$
$$ \geq \sum _{i=0}^{k-1}(d-1-i)+\sum _{i=k}^{a-1}(d-1-2i+k) = (+)$$
We have:
$$\sum _{i=k}^{a-1}(d-1-2i+k) = (a-k)(d-a)\,\,\,\,(*)$$
If $k=0$, we get $\deg Z \geq a(d-a)$, a contradiction since $\deg Z \leq a(d-a)-1$ by assumption. Assume $k>0$. Then:
$$\sum _{i=0}^{k-1}(d-1-i) = k(d-1)-\frac{k(k-1)}{2} = k(d-1 -\frac{(k-1)}{2})$$
From $(+)$ and $(*)$ we get:
$$\deg Z \geq (a-k)(d-a)+k(d-1-\frac{(k-1)}{2}) = a(d-a)+k(a-1-\frac{(k-1)}{2})$$
and to conclude it is enough to check that $a-1 \geq (k-1)/2$. Since $r \geq k$, this will follow from $a-1 \geq (r-1)/2$. If $a < (r+1)/2$, then $n_{a-1} = d-a+r > d+a-1 \geq d$, in contradiction with $n_{a-1} < d-2$. The proof is over. 
\end{proof}

We can now conclude and get the ``gaps part'' of \ref{Thm2}:

\begin{corollary}
\label{GAPS t>4}
For $c \geq 4$ let $y \in A_t$ for some $t$, $2 \leq t \leq c/2$. If $y$ is non admissible, then $y$ is a gap (i.e. $(c,y)$ is not effective).
\end{corollary}

\begin{proof} Since $y$ is non-admissible, $y \in G_t(a)$ for some $a \geq 1$ (see \ref{c(t-1)-y<0}),  or equivalently $\deg Z = c(t-1)-y \in [(a-1)(t-1)+1, a(t-1-a)-1]$ for some $a \geq 1$ such that $a^2+1\leq t-1$. In view of \ref{Z CBt-4} it is enough to show that $Z$ cannot verify Cayley-Bacharach for $t-4$. For this we want to apply \ref{Exist Gaps}. The only thing we have to show is $h^0(\ic _Z(a-1))=0$. Let $P$ be a curve of degree $\sG < a$ containing $Z$. If $P$ doesn't have a common component with some curve of $H^0(\ic _Z(t-1))$, then $\deg Z \leq \sG (t-1) \leq (a-1)(t-1)$. But this is impossible since $\deg Z \geq (a-1)(t-1)+1$. On the other hand $Z$ is contained in a pencil $\lag T,T' \rag$ of curves of degree $t-1$ and this pencil has a base locus of dimension zero (see \ref{Z CBt-4}). So we may always find a curve in $H^0(\ic _Z(t-1))$ having no common component with $P$.
\end{proof}

\section{The proof (existence).}

In this section we assume that $y \in A_t$ is admissible and prove that $y$ is indeed effective. Since $y$ is admissible we know by \cite{Co} that there exists a smooth plane curve, $T$, of degree $t-1$ and a globally generated line bundle, $\lc$, on $T$ of degree $z:=c(t-1)-y$.

\begin{lemma}
\label{L*(c) epsg}
Assume $y \in A_t$ is admissible. If $T$ is a smooth plane curve of degree $t-1$ and if $\lc$ is a globally generated line bundle on $T$ with $\deg \lc =c(t-1)-y$, then $\lc ^*(c)$ is non special and globally generated.
\end{lemma}

\begin{proof} We have $\deg \lc ^*(c)=y$. It is enough to check that $y \geq 2g_T+1 = (t-2)(t-3)+1$. We have $y \geq (t-1)(c-t+1)+1$. Since $c \geq 2t$ it follows that $y \geq (t-1)(t+1)+1 = t^2$.
\end{proof}

\begin{lemma}
\label{h1L>0}
Assume $y \in A_t$ is admissible. If there exists a smooth plane curve, $T$, of degree $t-1$, carrying a globally generated line bundle, $\lc$, with $\deg \lc = c(t-1)-y$ and with $h^1(\lc )\neq 0$, then $y$ is effective.
\end{lemma}

\begin{proof} Let $Z$ be a section of $\lc$. If $h^1(\lc )=h^0(\lc ^*(t-4)) \neq 0$, then $Z$ lies on a curve, $R$, of degree $t-4$. Set $X = T\cap R$. By \ref{L*(c) epsg} $\lc ^*(c)$ is globally generated, so we may find a $s \in H^0(\lc ^*(c))$ such that $(s)_0 \cap X=\emptyset$. Set $Y = (s)_0$. We have $\oc _T(c) \simeq \oc _T(Z+Y)$ and $Y \cap Z=\emptyset$. So $Y$ and $Z$ are linked by a complete intersection $I = F_c\cap T$. Let's prove that $Y$ satisfies $CB(c-3)$. First observe that there exists a degree $t-1$ curve, $T'$, containing $Z$ such that $T' \cap Y = \emptyset$: indeed since $Y \cap X=\emptyset$, we just take $T' =R\cup C$ where $C$ is a suitable cubic. Now let $p \in Y$ and let $P$ be a degree $c-3$ curve containing $Y'=Y\setminus \{p\}$. The curve $T'P$ contains $I\setminus \{p\}$ and has degree $c+t-4$. Since the complete intersection $I$ satisfies $CB(c+t-4)$ and since $T' \cap Y=\emptyset$, $p\in P$.
\par
It follows that we have: $0 \to \oc \to F \to \ic _Y(c) \to 0$ where $F$ is a rank two vector bundle with Chern classes $(c,y)$. Since $\ic _{Y,T}(c) \simeq \lc$ is globally generated, $\ic _Y(c)$ and therefore $F$ are globally generated.
\end{proof}

We need a lemma:

\begin{lemma}
\label{R gene}
For any integer $r$, $1 \leq r \leq h^0(\oc (t-1))-3$, there exists a smooth zero-dimensional subscheme, $R$, of degree $r$ such that $\ic _R(t-1)$ is globally generated with $h^0(\ic _R(t-1)) \geq 3$.
\end{lemma}

\begin{proof} Take $R$ of degree $r$, of maximal rank. If $h^0(\oc (t-2)) \geq r$, then $h^1(\ic _R(t-2))=0$ and we conclude by Castelnuovo-Mumford's lemma. Assume $h^0(\oc (t-2)) < r$ and take $R$ of maximal rank and minimally generated (i.e. all the maps $\sG (m):H^0(\ic _R(m)) \otimes H^0(\oc (1)) \to H^0(\ic _R(m+1))$ are of maximal rank). If $\sG (t-1)$ is surjective we are done, otherwise it is injective and the minimal free resolution looks like:
$$0 \to d.\oc (-t-1) \to b.\oc (-t) \oplus a.\oc (-t+1) \to \ic _R \to 0$$
By assumption $a \geq 3$. 
\par
Since $\mathcal{H}om(d-\oc (-t-1),b.\oc (-t)\oplus a.\oc (-t+1))$ is globally generated, if $\fG \in Hom(d.\oc (-t-1),b.\oc (-t)\oplus a.\oc (-t+1))$ is sufficiently general, then $Coker(u) \simeq \ic _R$ with $R$ smooth of codimension two. Furthermore since $b.\oc (1)$ is globally generated, it can be generated by $b+2$ sections; it follows that the general morphism $f: d.\oc \to b.\oc (1)$ is surjective ($d=a+b-1 \geq b+2$). In conclusion the general morphism $\fG = (f,g): d.\oc (-t-1) \to b.\oc (-t)\oplus a.\oc (-t+1)$ has $Coker(\fG )\simeq \ic _R$ with $R$ smooth, with the induced morphism $a.\oc (-t+1) \to \ic _R$ surjective.
\end{proof}  

\begin{proposition}
\label{Existence}
Let $c \geq 4$ be an integer. For every $2 \leq t \leq c/2$, every admissible $y \in A_t$ is effective.
\end{proposition}

\begin{proof} By \cite{Co} there exists a globally generated line bundle, $\lc$, of degree $l = c(t-1)-y$ on a smooth plane curve, $T$, of degree $t-1$. If $h^1(\lc )\neq 0$ we conclude with \ref{h1L>0}. Assume $h^1(\lc )=0$. Then $h^0(\lc )= l-g_T+1 \geq 2$ (we may assume $\lc \neq \oc _T$, because if $y =c(t-1)$, we are done). So $l \geq \frac{(t-2)(t-3)}{2}+1$. Since $(t-1)(c-t+1)+1 \leq y \leq t(c-t)-1$, we have:
$$(t-1)^2-1 \geq l \geq \frac{(t-2)(t-3)}{2}+1\,\,\,(*)$$
It follows that: 
$$l = (t-1)^2-r,\,\,\,1 \leq r \leq \frac{t(t+1)}{2}-3 = h^0(\oc (t-1))-3\,\,\,(**)$$
\par
For $r$, $1 \leq r \leq h^0(\oc (t-1))-3$, let $R \subset \Ptw$ be a general set of $r$ points of maximal rank, with $h^0(\ic _R(t-1))\geq 3$ and $\ic _R(t-1)$ globally generated (see \ref{R gene}). It follows that $R$ is linked by a complete intersection $T \cap T'$ of two smooth curves of degree $t-1$, to a set, $Z$, of $(t-1)^2-r =l$ points. Since $\ic _R(t-1)$ is globally generated, $\ic _{R,T}(t-1) \simeq \oc _T(t-1-R)$ is globally generated. Since $\oc _T(t-1) \simeq \oc _T(R+Z)$, we see that $\lc := \oc _T(Z)$ is globally generated. Moreover, by construction, $h^0(\ic _Z(t-1))\geq 2$. By \ref{L*(c) epsg}, $\lc ^*(c)$ is globally generated so there exists $s \in H^0(\lc ^*(c))$ such that: $Y:= (s)_0$ satsfies $Y \cap (T\cap T')=\emptyset$. As in the proof of \ref{h1L>0}, we see that $Y$ satisfies $CB(c-3)$: indeed $T'$ is a degree $t-1$ curve containing $Z$ such that $T'\cap Y=\emptyset$. Since $\ic _{Y,T}(c) \simeq \lc$ is globally generated, we conclude that $\ic _Y(c)$ is globally generated.
\end{proof}

\par
Proposition \ref{Existence} and Corollary \ref{GAPS t>4} (and Remark \ref{c<4})  prove Theorem \ref{Thm2}. It follows that the proof of Theorem \ref{Thm} is complete.

\section{Morphisms from $\Ptw$ to $G(1,3)$.} \label{S: G}

\par
It is well known that finite morphisms $\fG : \Ptw \to G(1,3)$ are in bijective correspondence with exact sequences of vector bundles on $\Ptw$:
$$0 \to E^* \to 4.\oc \to F \to 0 \,\,\,\,\,(*)$$
where $F$ has rank two and is globally generated with $c_1(F)=c>0$. If $\fG$ is generically injective, then $\fG (\Ptw) = S \subset G \subset \Pcq$ (the last inclusion is given by the Pl\H{u}cker embedding) has degree $c^2$ (as a surface of $\Pcq$) and bidegree $(y, c^2-y)$, $y = c_2(F)$ (i.e. there are $y$ lines of $S$ through a general point of $\Pt$ and $c^2-y$ lines of $S$ contained in a general plane of $\Pt$). Theorem \ref{Thm} gives all the possible $(c,y)$ (but it doesn't tell if $\fG$ exists). Finally, by \cite{T}, if $\fG$ is an embedding then $(c,y) \in \{(1,0), (1,1), (2,1), (2,3) \}$.


\end{document}